\documentclass[11pt]{article}
\usepackage{graphicx}

\usepackage{amsmath, amsthm, amscd, amsfonts, amssymb}

\usepackage{a4} 
\usepackage{caption,subcaption}
\usepackage{multirow}
\usepackage{xstring}
\usepackage[utf8]{inputenc}
\usepackage{amstext,amsbsy,amsopn,eucal,enumerate}
\usepackage{tikz}
\usepackage{pgfplots}
\usepackage{tabularx}
\usepackage{mathtools}
\usepackage{todonotes}

\newcommand{\diag}{\operatorname{diag}}
\DeclareMathOperator{\diff}{d}

\newcommand{\ds}{{\diff}s}

\newcommand{\dt}{{\diff}t}

\newcommand{\dw}{{\diff}w}

\newcommand{\expn}{\operatorname{e}}

\newcommand{\spaned}{\operatorname{span}}
\newcommand{\kernel}{\operatorname{ker}}
\newcommand{\vect}{\operatorname{vec}}
\newcommand{\im}{\operatorname{im}}

\newcommand {\smat}      [1] {\left(\begin{smallmatrix}{#1}}
\newcommand {\srix}          {\end{smallmatrix}\right)}
\newcommand {\s}      [1] {\begin{smallmatrix}{#1}}
\newcommand {\se}          {\end{smallmatrix}}
\newcommand{\trace}{\operatorname{tr}}

\newtheorem{defn}{Definition}[section]

\newtheorem{prop}[defn]{Proposition}  

\newtheorem{thm}[defn]{Theorem}

\newtheorem{remark}[defn]{Remark}

\title{Gramian-based model reduction for unstable stochastic systems}
\author{
	Martin Redmann\thanks{Martin Luther University Halle-Wittenberg, Institute of Mathematics, Theodor-Lieser-Str. 5, 06120 Halle (Saale), Germany (Emails: {\tt 
			martin.redmann@mathematik.uni-halle.de, Nahid.Jamshidi@mathematik.uni-halle.de}), 
			}\and Nahid Jamshidi$^*$
}

\begin{document}
	
	\maketitle
	
	\begin{abstract}
This paper considers large-scale linear stochastic systems representing, e.g., spatially discretized stochastic partial differential equations. Since asymptotic stability can often not be ensured in such a stochastic setting (e.g. due to larger noise), the main focus is on establishing model order reduction (MOR) schemes applicable to unstable systems. MOR is vital to reduce the dimension of the problem in order to lower the enormous computational complexity of for instance sampling methods in high dimensions. In particular, a new type of Gramian-based MOR approach is proposed in this paper that can be used in very general settings. The considered Gramians are constructed to identify dominant subspaces of the stochastic system as pointed out in this work. Moreover, they can be computed via Lyapunov equations. However, covariance information of the underlying systems enters these equations which is not directly available. Therefore, efficient sampling based methods relying on variance reduction techniques are established to derive the required covariances and hence the Gramians. Alternatively, an ansatz to compute the Gramians by deterministic approximations of covariance functions is investigated. An error bound for the studied MOR methods is proved yielding an a-priori criterion for the choice of the reduced system dimension. This bound is new and beneficial even in the deterministic case. The paper is concluded by numerical experiments showing the efficiency of the proposed MOR schemes.
	\end{abstract}

	

	\section{Introduction}\label{stochstabgen}
	
	Let $w=\left(w_1, \ldots, w_q\right)^\top$ be an $\mathbb R^q$-valued mean zero Wiener process with covariance matrix $\mathbf K=(k_{ij})$, i.e., $\mathbb E[w(t)w^\top(t)]=\mathbf K t$ for $t\in [0, T]$, 
	where $T>0$ is the terminal time.
	Suppose that $W$ and all stochastic process appearing in this paper are defined on a filtered probability space  $\left(\Omega, \mathcal F, (\mathcal F_t)_{t\in [0, T]}, 
	\mathbb P\right)$\footnote{$(\mathcal F_t)_{t\in [0, T]}$ shall be right continuous and complete.}. In addition, we assume $w$ to be $(\mathcal F_t)_{t\in [0, T]}$-adapted and the 
	increments $w(t+h)-w(t)$ to be independent of $\mathcal F_t$ for $t, h\geq 0$. We consider the following large-scale controlled linear stochastic differential equation
	\begin{subequations}\label{original_system}
		\begin{align}\label{stochstatenew}
			dx(t)&=[Ax(t) + B u(t)] dt+\sum_{i=1}^q N_i x(t) dw_i(t),\quad x(0)=x_0,\\ \label{output}
			y(t)&= C x(t), \quad t\in[0, T],
		\end{align}
	\end{subequations}
	where $A, N_i\in \mathbb R^{n\times n}$, $B\in \mathbb R^{n\times m}$ and $C\in \mathbb R^{p\times n}$. The state dimension $n$ is assumed to be large and the quantity of interest $y$ is often low-dimensional, 
	i.e., $p\ll n$, but we also discuss the case of a large $p$. By $x(t; x_0, u)$, we denote the state in dependence on the initial state $x_0$ and the control $u$, for which we assume that it is $(\mathcal F_t)_{t\in [0, T]}$-adapted and $\left\|u\right\|_{L^2_T}^2:=\mathbb E\int_0^T \left\|u(s)\right\|^2_2 ds<\infty$ with $\left\|\cdot\right\|_2$ representing the Euclidean norm.\smallskip
	
	The goal is to construct a system with state $\bar x$ and quantity of interest $\bar y$ having the same structure as \eqref{original_system} but a much smaller state dimension $r\ll n$. At the same time, it is aimed to  ensure $y \approx \bar y$. Such a reduced order model (ROM) is particularly beneficial if many evaluations \eqref{original_system} for several controls $u$ are required (e.g. in an optimal control problem) combined with need of generating many samples of $y$ for each individual $u$. Now, a ROM shall be achieved under very general conditions such as the absence of mean square asymptotic stability, i.e., $\mathbb E\left\|x(t; x_0, 0)\right\|_2^2\rightarrow 0$ (as $t\rightarrow \infty$) is not given. Methods involving such a stability condition are intensively studied in the literature \cite{beckerhartmann, redmannbenner, redmannspa2, redmannigor} since it is often guaranteed if \eqref{stochstatenew} results from a spatial discretization of a stochastic partial differential equation (SPDE) such as \begin{align}\label{heateq}
			\frac{\partial {\mathcal X}(t,\zeta)}{\partial t}=\Delta \mathcal X(t,\zeta)+\mathcal B u(t)+\sum_{i=1}^q\mathcal N_i \mathcal X(t,\zeta)\frac{\partial w_i(t)}{\partial t}.                                                                                                                                                                                                              \end{align}
We refer to \cite{dapratozab} for more details on the theory of such equations. 
The solution ${\mathcal X}(t, \cdot)$ to the heat equation \eqref{heateq} is viewed as a stochastic process taking values in a Hilbert space and shall be approximated by $x$. In this context, $A$ can be seen as a discretized version of the Laplacian $\Delta$ and $B$, $N_i$ represent discretizations of the linear bounded operators $\mathcal B$, $\mathcal N_i$. Moreover, $w_i$ can be interpreted as Fourier coefficients corresponding to a truncated series of space-time noise. Further explanations on different schemes for a spatial discretization can, e.g., be found in \cite{Barth, galerkinhaus}. However, even in a setting like in \eqref{heateq}, mean square asymptotic stability can be violated since the noise can easily cause instabilities (e.g. if it is sufficiently large).\smallskip

Such a scenario is of interest in this paper. We establish generalizations of balancing related model order reduction (MOR) schemes in order to make them applicable to general systems \eqref{original_system}. These MOR methods rely on matrices called Gramians that can be used to identify  the dominant subspaces of \eqref{original_system}. Based on this characterization of the relevance of different state directions, less important information in the dynamics is removed leading to the desired ROM. This step can be interpreted as an optimization procedure applied to  spatially discretized SPDE. In an unstable setting, Gramians need to be defined that generally exist in contrast to previous approaches. We consider generalized time-limited Gramians in this work. Such type of Gramians have been used in deterministic frameworks \cite{morGawJ90, kurschner2018balanced, redmannkuerschner}. Although such an ansatz is beneficial for the setting we want to cover, the analysis of MOR methods based on generalized time-limited Gramians is much more challenging. Furthermore, the question of how to compute these Gramians in practice is very difficult but vital since they are required to derive the ROM. \smallskip

In this paper, we introduce time-limited Gramian in the stochastic setting studied here. We point out the relation between these Gramians and the dominant subspaces of \eqref{original_system} and show their relation to matrix (differential) equations. Subsequently, we discuss two different MOR techniques based on these Gramians and analyze the respective error. In particular, an error bound is established that allows us to identify situations in which the approaches work well. It is important to mention that this bound is more than just a generalization of the deterministic case \cite{redmannkuerschner}. The new type of representation links the truncated Hankel singular values of the system or the truncated eigenvalues of the reachability Gramian, respectively, to the error of the approximation without needing asymptotic stability and is hence beneficial also in unstable settings. Moreover, we discuss different strategies that can be used to compute the proposed Gramians. They are solutions to Lyapunov equations. However, in a time-limited scenario, covariance information at the terminal time enters these Lyapunov equations which is not immediately available. Since direct methods only work in moderate high dimensions, we focus on sampling based approaches to estimate the required covariances. In order to increase the efficiency of such procedures we apply variance reduction methods in this context leading to an efficient way of solving for the time-limited Gramians. Apart from this empirical procedure, a second strategy to approximate covariance functions and hence the Gramians is investigated, where potentially expensive sampling is not required. The paper is concluded by several numerical experiments showing the efficiency of the MOR methods.

	\section{Gramian-based MOR}\label{sec_mor}
	
	\subsection{Gramians and characterization of dominant subspaces}\label{sec:gram}
	Identifying the effective dimensionality of system \eqref{original_system} requires the study of 
	the fundamental solution to the homogeneous stochastic state equation. It is defined as the matrix valued stochastic process $\Phi$ solving
	\begin{align}\label{funddef}
		\Phi(t)=I+\int_0^t A \Phi(s) ds+\sum_{i=1}^q \int_0^t  N_i \Phi(s)dw_i(s), \quad t\in [0, T],
	\end{align}
	where $I$ denotes the identity matrix. Multiplying \eqref{funddef} with $x_0$ from the right, we obtain the solution to \eqref{stochstatenew} if $u\equiv 0$. Based on $\Phi$ we define two Gramians by
	\begin{align}\label{reach_gram}
		P_{T}:=&\mathbb E\int_0^{T} \Phi(s)BB^\top \Phi^\top(s) ds\\
		Q_{T}:=&\mathbb E\int_0^{T} \Phi^\top(s)C^\top C \Phi(s) ds,
	\end{align}
	where $P_T$ and $Q_T$ are supposed to identify the less relevant states in \eqref{stochstatenew} and \eqref{output}, respectively. $P_T$ and $Q_T$ can be viewed as generalizations of deterministic time-limited Gramians which are obtained by setting $N_i=0$ for all $i=1, \ldots, q$ resulting in $\Phi(t)= \expn^{At}$. MOR schemes based on such Gramians in a deterministic framework are investigated, e.g., in \cite{morGawJ90, kurschner2018balanced, redmannkuerschner}. $P_T$ and $Q_T$ generally exist in contrast to their limits $\lim_{T\rightarrow \infty} P_{T}$ and $\lim_{T\rightarrow \infty}Q_{T}$ which require mean square asymptotic stability. MOR methods based on these limits are, e.g., considered in \cite{beckerhartmann, redmannbenner, redmannspa2, redmannigor} and are already analyzed in detail. However, the necessary stability condition is often not satisfied in practice. \smallskip
	
	Let us briefly sketch the relation between $P_T$ and dominant subspaces in \eqref{stochstatenew} for the case of zero initial data. Suppose that  $(p_{k})_{k=1,\ldots, n}$ is an orthonormal basis of $\mathbb R^n$ consisting of eigenvectors of $P_T$. We can then write the state as \begin{align*}
		x(t; 0, u)=\sum_{k=1}^n  \left\langle x(t; 0, u), p_{k} \right\rangle_2 p_{k}.  \end{align*}
	Given $x_0=0$, the expansion coefficient can be bound from above as follows \begin{align}\label{diffreachjanein}
		\sup_{t\in[0, T]}\mathbb E \left\vert\langle x(t, 0, u), p_{k}  \rangle_2\right\vert \leq  \sqrt{\lambda_{k}}\left\|u\right\|_{L^2_T},
	\end{align}
	see \cite[Section 3]{redmannspa2}, 
	where $\lambda_{k}$ is the eigenvalue corresponding to $p_k$. If $\lambda_{k}$ is small, the same is true for $\langle x(\cdot, 0, u), p_{k}  \rangle_2$ and hence $p_k$ is a less relevant direction that can be neglected. 
	This implies that the eigenspaces of $P_T$ belonging to the small eigenvalues can be removed from the system. On the other hand, we aim to find state directions that have a low impact on the quantity of interest $y$. We therefore look at the initial state $x_0$ since it determines the dynamics of the state variable. We expand \begin{align*}
		x_0=\sum_{k=1}^n  \left\langle x_0, q_{k} \right\rangle_2 q_{k}, \end{align*}
	where $(q_{k})_{k=1,\ldots, n}$ is an orthonormal basis of eigenvectors of $Q_T$ with associated eigenvalues $(\mu_{k})_{k=1,\ldots, n}$. Using the solution representation of the state variable, we obtain \begin{align*}	                                                                                                                                                                                                              
     y(t; x_0, u)  &= C\Phi(t) x_0 + C\int_0^t \Phi(t, s) B u(s) ds\\  
     &=  \sum_{k=1}^n  \left\langle x_0, q_{k} \right\rangle_2 C\Phi(t)q_{k} + C\int_0^t \Phi(t, s) B u(s) ds\end{align*}
with $t\in[0, T]$ and $ \Phi(t, s):=\Phi(t)\Phi^{-1}(s)$. Consequently, neglecting $q_k$ has a low impact on $y$ if  $C\Phi(\cdot)q_{k}$ is small on $[0, T]$. It now follows that 
	\begin{align}\label{outenergiewithremainder}
		\mathbb{E} \int_0^T\left\| C\Phi(t)q_{k}\right\|_2^2 dt = q_k^\top Q_T q_k= \mu_k,   
	\end{align}
	telling us that the eigenspaces of $Q_T$ are unimportant for which the associated eigenvalues $\mu_k$ are small. Knowing both the less relevant state directions in \eqref{stochstatenew} and \eqref{output} from \eqref{diffreachjanein} and \eqref{outenergiewithremainder} it is aimed to remove them. This can be done by diagonalizing $P_T$ such that less important variables in \eqref{stochstatenew} can be easily identified and truncated. Another, but computationally more expensive, approach is based on simultaneously diagonalizing $P_T$ and $Q_T$ which allows to remove more redundant information from the system. Both strategies are discussed in Section \ref{sec_diagonaling}. \smallskip
	
	Below, we point out the relation between the Gramians and linear matrix differential equations. To do so, we introduce two operators $\mathcal L_A(X)= A X+X A^\top$ and $\Pi(X) = \sum_{i, j=1}^q N_i X N_j^\top k_{ij}$ on the space of symmetric matrices endowed with the Frobenius inner product $\langle \cdot, \cdot \rangle_F$. $\mathcal L_A$ is a Lyapunov operator and $\Pi$ is positive in the sense that $\Pi(X)$ is a positive semidefinite matrix if $X$ is positive semidefinite. The corresponding adjoint operators are $\mathcal L_A^*(X)= A^\top X+X A$ and $\Pi^*(X) = \sum_{i, j=1}^q N_i^\top X N_j k_{ij}$.
	
	The equations related to $P_T$ and $Q_T$ will be helpful to compute these Gramians that are needed in order to derive the reduced system. By Ito's product rule \cite{oksendal}, we can show that $F(t) = \mathbb E [\Phi(t)BB^\top \Phi^\top(t)]$, $t\in [0, T]$, solves
	\begin{align}\label{matrixequalforF}
		{\dot F}(t) = \mathcal L_A\left(F(t)\right)+ \Pi\left(F(t)\right),\quad F(0)=BB^\top.
	\end{align}
	Integrating both sides of (\ref{matrixequalforF}) yields 
\begin{align}\label{comPT}
		F(T)-BB^\top = 
		\mathcal L_A\left(P_T\right) + \Pi\left(P_T\right),
	\end{align}
	see \cite{damm, redmannspa2, mor_heston}.
	\begin{remark}\label{rem_vect}
The generalized Lyapunov operator $\mathcal L_A + \Pi$ is linked to the Kronecker matrix \begin{align}\label{stochstab}
		\mathcal K= A\otimes I+I\otimes A+\sum_{i, j=1}^q N_i\otimes N_j k_{ij},\end{align}
where $\cdot\otimes \cdot$ is the Kronecker product between two matrices. Let $\vect(\cdot)$ be the vectorization of a matrix. Then, it holds that 
$\vect\left((\mathcal L_A + \Pi)\left(X\right)\right)= \mathcal K \vect(X)$.
\end{remark}

	The link between $Q_T$ and the corresponding matrix equation is established in a different way. We formulate this result in the following proposition.
	\begin{prop}\label{matrix_ode_Q}
	 Let $C^\top C$ be contained in the eigenspace of the Lyapunov operator $\mathcal L_A^* + \Pi^*$. Then, $G(t) = \mathbb E [\Phi^\top(t)C^\top C \Phi(t)]$, $t\in [0, T]$, satisfies
	 \begin{align}\label{ODEforG}
		\dot G(t)= \mathcal L_A^*\left(G(t)\right) + \Pi^*\left(G(t)\right), \quad G(0)=C^\top C.
	\end{align}
	\end{prop}
	\begin{proof}
	 Since $C^\top C$ is contained in the eigenspace of the Lyapunov operator, there exist $\alpha_1, \dots, \alpha_{n^2}\in\mathbb C$ such that $C^\top C = \sum_{k=1}^{n^2} \alpha_k \mathcal V_k$, where $(\mathcal V_k)$ are eigenvectors of $\mathcal L_A^* + \Pi^*$ corresponding to the eigenvalues $(\beta_k)$. Then, we have $\mathbb E [\Phi^\top(t)C^\top C \Phi(t)]= \sum_{k=1}^{n^2} \alpha_k \mathbb E [\Phi^\top(t)\mathcal V_k\Phi(t)]$. Let us apply Ito's product rule, see \cite{oksendal}, to $\Phi^\top(t)\mathcal V_k\Phi(t)$ resulting in \begin{align*} 	                                                                                                                                                                                                                                                                                                                                                                                                                                                                                                                                      
  d\left(\Phi^\top(t)\mathcal V_k\Phi(t)\right) =   d\left(\Phi^\top(t)\right)\mathcal V_k\Phi(t) + \Phi^\top(t)\mathcal V_k d\left(\Phi(t)\right) +   d\left(\Phi^\top(t)\right)\mathcal V_k d\left(\Phi(t)\right).                                                                                                                                                                                                                                       \end{align*}
We insert the stochastic differential of $\Phi$ above, compare with \eqref{funddef}, leading to \begin{align*} 	                                                                                                                                                                                                                                                                                                                                                                                                                                                                                                                                      
  d\left(\Phi^\top(t)\mathcal V_k\Phi(t)\right) =&   \left(\Phi^\top(t)A^\top dt+\sum_{i=1}^q \Phi^\top(t)N_i^\top dw_i(t)\right)\mathcal V_k\Phi(t) \\
  &+ \Phi^\top(t)\mathcal V_k \left(A\Phi(t) dt+\sum_{i=1}^q N_i\Phi(t) dw_i(t)\right) +   \Phi^\top(t) \sum_{i, j=1}^q  N_i^\top \mathcal V_k N_j k_{ij}\Phi(t) dt\\
  =&  \;\Phi^\top(t)\left(A^\top \mathcal V_k + \mathcal V_k A +\sum_{i, j=1}^q  N_i^\top \mathcal V_k N_j k_{ij}\right)\Phi(t) dt\\
  &+\sum_{i=1}^q \Phi^\top(t)\left(N_i^\top \mathcal V_k +\mathcal V_k  N_i\right) \Phi(t)dw_i(t).
  \end{align*}
  We apply the expected value to both sides of the above identity and exploit that Ito integrals have mean zero (see e.g. \cite{oksendal}). Hence, we obtain \begin{align*} 	                                                                                                                                                                                                                                                                                                                                                                                                                                                                                                                                      
  \frac{d}{dt}\mathbb E [\Phi^\top(t)\mathcal V_k\Phi(t)]= \mathbb E [\Phi^\top(t)(\mathcal L_A^*+ \Pi^*)(\mathcal V_k)\Phi(t)] = \beta_k \mathbb E [\Phi^\top(t)\mathcal V_k\Phi(t)].
  \end{align*}
This implies that $\mathbb E [\Phi^\top(t)\mathcal V_k\Phi(t)] = \expn^{\beta_k t}\mathcal V_k$ providing $\mathbb E [\Phi^\top(t)C^\top C \Phi(t)]= \sum_{k=1}^{n^2} \alpha_k \expn^{\beta_k t}\mathcal V_k$. Consequently, we have 
\begin{align*} 	                                                                                                                                                                                                                                                                                                                                                                                                                                                                                                                                      
  \frac{d}{dt}\mathbb E [\Phi^\top(t) C^\top C \Phi(t)]&=  \sum_{k=1}^{n^2} \alpha_k \expn^{\beta_k t}\beta_k \mathcal V_k=  \sum_{i=1}^{n^2} \alpha_k \expn^{\beta_k t}(\mathcal L_A^*+ \Pi^*)(\mathcal V_k)\\
  &= (\mathcal L_A^*+ \Pi^*)(\mathbb E [\Phi^\top(t) C^\top C \Phi(t)])
\end{align*}
using the linearity of $\mathcal L_A^*+ \Pi^*$. This concludes the proof.
  \end{proof}
\begin{remark}\label{remcomQ}
The assumption of Proposition \ref{matrix_ode_Q} is always true if $\mathcal K$ is diagonalizable over $\mathbb C$ because in that case there is a basis of $\mathbb C^{n^2}$ consisting of eigenvectors of $\mathcal K^\top$. Hence, $\vect(C^\top C)$ can be spanned by these eigenvectors which are of the form $\vect(\mathcal V_k)$ with $\mathcal V_k$ being an eigenvector of $\mathcal L_A^*+ \Pi^*$ providing that $C^\top C$ is in the eigenspaces of this operator. Therefore, from the computational point of view, the assumption of Proposition \ref{matrix_ode_Q} does not restrict the generality since the set of diagonalizable $n^2\times n^2$ matrices is dense in $\mathbb C^{n^2\times n^2}$.\\
In fact, we can find a stochastic representation of the solution to \eqref{ODEforG} different from $\mathbb E [\Phi^\top(t)C^\top C \Phi(t)]$, $t\in [0, T]$. Introducing the fundamental solution $\Phi_d$ by the equation $\Phi_d(t)=I+\int_0^t A^\top \Phi_d(s) ds+\sum_{i=1}^q \int_0^t  N_i^\top \Phi_d(s)dw_i(s)$, we see that $G(t) = \mathbb E [\Phi_d(t)C^\top C \Phi_d^\top(t)]$. This is a direct consequence of the relation between $\mathbb E [\Phi(t)B B^\top \Phi^\top(t)]$ and the solution of \eqref{matrixequalforF} when $(A, B, N_i)$ is replaced by $(A^\top, C^\top, N_i^\top)$. Therefore, $\mathbb E [\Phi_d(t)C^\top C \Phi_d^\top(t)]$, $t\in [0, T]$, solves \eqref{ODEforG} and hence coincides with $\mathbb E [\Phi^\top(t)C^\top C \Phi(t)]$, $t\in [0, T]$, given the assumption of Proposition \ref{matrix_ode_Q}.\\
Generally, we have $\Phi_d(t)\neq \Phi^\top(t)$. In case all matrices $A, N_1, \ldots, N_q$ commute, we know that $A$ and $N_i$ commute with $\Phi$ (see, e.g., \cite{mor_heston}). Hence, $\Phi_d(t)= \Phi^\top(t)$ which can be seen be transposing \eqref{funddef} and subsequently exploiting the commutative property. This is particularly given in the deterministic case where $N_i=0$ for all $i=1, \ldots, q$.
\end{remark}
Under the assumption of Proposition \ref{matrix_ode_Q}, it holds that 
\begin{align}\label{gengenlyapobs}
		G(T)-C^\top C = \mathcal L^*_A\left(Q_T\right) + \Pi^*\left(Q_T\right), 
	\end{align} 
	exploiting \eqref{ODEforG}. In fact, we need to compute $P_T$ and $Q_T$ within the MOR procedure described later.  Lyapunov equations \eqref{comPT} and \eqref{gengenlyapobs} are used to do so. However, one needs to have access to $F(T)$ and $G(T)$ which are the terminal values of the matrix-differential equations \eqref{matrixequalforF} and \eqref{ODEforG}. This is indeed very challenging in a framework, where $n\gg 100$. We will address possible approaches for computing $P_T$ and $Q_T$ for such settings in Section \ref{sec_compGram}.

	\subsection{Reduced order modeling by transformation of Gramians}\label{sec_diagonaling}
	
In this work, we address MOR techniques that rely on a change of basis. In particular, one seeks for a suitable regular matrix $S$ that defines $x_S(t) = S x(t)$. Inserting this into \eqref{original_system} yields
		\begin{align}\label{stochstate_trans}
			dx_S(t) = [A_S x_S(t) + B_S u(t)] dt+\sum_{i=1}^q N_{i, S} x_S(t) dw_i(t),\quad
			y(t) = C_S x_S(t), \quad t\in[0, T], 
		\end{align}
where $(A_S, B_S, C_S, N_{i, S}) = (SAS^{-1},SB, CS^{-1}, SN_iS^{-1})$.
System \eqref{stochstate_trans} has the same input-output behavior as \eqref{original_system} but the fundamental solution and hence the Gramians are different. The fundamental solution of \eqref{stochstate_trans} is $\Phi_S(t) = S \Phi(t) S^{-1}$ which can be observed by multiplying \eqref{funddef} with $S$ from the left and with $S^{-1}$ from the right. Consequently, the new Gramians are \begin{align*}                                                                                                                                                                                                                                      
                                                                                                                                                                                                                                                                                                                                                                                                                                                                                                                                                                                                                                                           P_{T, S}=\mathbb E\int_0^{T} \Phi_S(s)B_SB_S^\top \Phi_S^\top(s) ds = S P_T S^\top,\quad 
		Q_{T, S}=\mathbb E\int_0^{T} \Phi_S^\top(s)C_S^\top C_S \Phi_S(s) ds =S^{-\top} Q_T S^{-1}. \end{align*}
The idea is to diagonalize at least one of these Gramians, since in a system with diagonal Gramians, the orthonormal bases $(p_k)$ and $(q_k)$ are canonical unit vectors (columns of the identity matrix). Thus, unimportant directions can be identified easily by \eqref{diffreachjanein} and \eqref{outenergiewithremainder} and are associated to the small diagonal entries of the new Gramians. For the first approach, we set $S=S_1$, where $S_1$ is part of the eigenvalue decomposition $ P_T= S_1^\top \Sigma^{(1)}_T S_1$. This leads to $P_{T, S}=\Sigma^{(1)}_T$ with $\Sigma^{(1)}_T$ being the diagonal matrix of eigenvalues of $P_T$. Notice that $S^\top = S^{-1}$ holds in this case. If \eqref{stochstatenew} is mean square asymptotically stable, $P_T$ can be replaced by $\lim_{T\rightarrow \infty} P_T$. This method based on the limit is investigated in \cite{redmannigor}.\smallskip

The second approach uses $S=S_2$, which leads to $P_{T, S}=Q_{T, S}=\Sigma^{(2)}_T$, where $\Sigma^{(2)}_T$ is the diagonal matrix of the square roots of eigenvalues of $P_T Q_T$. Those are called Hankel singular values (HSVs). Given $P_T, Q_T>0$, the transformation $S_2$ and its inverse are obtained by
\begin{align}\label{balancedtrans}
 S_2={\Sigma^{(2)}_T}^{-\frac{1}{2}}U^\top L^\top,\quad S_2^{-1}=K V{\Sigma^{(2)}_T}^{-\frac{1}{2}},
\end{align}
where the ingredients of \eqref{balancedtrans} are computed by the factorizations  $P_{T} = K K^\top$, $Q_T=LL^\top$ and the singular value decomposition of $K^\top L = V\Sigma^{(2)}_T U^\top$. The same procedure can be conducted for the limits of the Gramians (as $T\rightarrow \infty$) if mean square asymptotic stability is given \cite{redmannbenner}. However, such a stability condition is generally too restrictive in practice. We introduce the matrix \begin{align}
 \Sigma_T=\diag(\sigma_{T, 1},\ldots,\sigma_{T, n}) =   \Sigma^{(i)}_T, \quad i\in\{1, 2\},                                                                                                                                                                                                                                                                                                                                                                                                                                         \end{align}
as the diagonal matrix of either eigenvalues of $P_T$ or of HSVs of system \eqref{original_system}. For $S=S_1$ or $S=S_2$ the coefficients of \eqref{stochstate_trans} are partitioned as follows \begin{equation}\label{partitions}
\begin{aligned}
 A_S =& \begin{pmatrix}
        A_{11} & A_{12}\\
        A_{21}& A_{22}
       \end{pmatrix}, \quad B_S = \begin{pmatrix}
        B_{1} \\
        B_{2}
       \end{pmatrix},  \quad C_S = \begin{pmatrix}
        C_{1} & C_{2}
       \end{pmatrix},\\
   N_{i, S}=   & \begin{pmatrix}
        N_{i, 11} & N_{i, 12}\\
        N_{i, 21} & N_{i, 22}
       \end{pmatrix}, \quad x_S(t) = \begin{pmatrix}
        x_{1}(t) \\
        x_{2}(t)
       \end{pmatrix}, \quad \Sigma_T = \begin{pmatrix}
        \Sigma_{T, 1} & \\
        & \Sigma_{T, 2}
       \end{pmatrix},
\end{aligned}
\end{equation}
where $x_1(t)\in \mathbb{R}^{r}$, $A_{11}\in \mathbb{R}^{r\times r} $, $ B_1\in \mathbb{R}^{r\times m} $, $ C_1\in \mathbb{R}^{p\times r} $,  $ N_{i,11}\in \mathbb{R}^{r\times r}$ and $\Sigma_{T, 1}\in \mathbb{R}^{r\times r}$ etc. The variables $x_2$ are associated to the matrix $\Sigma_{T, 2}$ of small diagonal entries of $\Sigma_{T}$ and are the less relevant ones. A reduced system is now obtained by truncating the equations of $x_2$ in \eqref{stochstate_trans}. Additionally, we set $x_2\equiv 0$ in the equations for $x_1$ leading to a reduced system 
	\begin{subequations}\label{reduced_system}
		\begin{align}\label{stochstatenew2}
			d\bar{x}(t)&=[{A}_{11}\bar{x}(t) + {B}_1u(t)] \dt+\sum_{i=1}^q {N}_{i,11} \bar{x}(t) \dw_i(t),\quad \bar{x}(0)=\bar{x}_0,\\ \label{output2}
			\bar{y}(t)&= {C}_1\bar{x}(t), \quad t\in[0, T],
		\end{align}
	\end{subequations}
	approximating \eqref{original_system}. Below, we give another interpretation for \eqref{reduced_system}. Let us decompose the transformation \begin{align}\label{defineV}
	S= \begin{pmatrix}
			W^\top \\  \star 
		\end{pmatrix}, \quad
S^{-1}= \begin{pmatrix}
			V &  \star 
		\end{pmatrix}                                                                                                                                                        \end{align}
where $W^\top$ and $V$ are the first $r$ rows and columns of $S$ and $S^{-1}$, respectively. Notice that $W^\top V=I$ and hence $V W^\top$ is a projection. Furthermore, we have $W=V$ if $S=S_1$. Consequently, \eqref{reduced_system} can be seen as a projection-based model with ${A}_{11}= W^\top A V$, $B_1= W^\top B$, $C_1 = CV$ and ${N}_{i,11}= W^\top {N}_{i} V$ which is obtained by the state approximation $x(t) \approx V \bar x(t)$. Inserting this approximation into \eqref{original_system} and subsequently multiplying the state equation with $W^\top$ to enforce the remainder term to be zero then results in \eqref{reduced_system}.

\section{Output Error Bound }

In this section, we prove a bound for the error between \eqref{original_system} and \eqref{reduced_system}. Below, we assume zero initial conditions, i.e., $x_0=0$ and $\bar x_0=0$. We begin with a general bound following the steps of \cite{redmannbenner, redmannspa2}.
The solutions $x(t)$ and $\bar{x}(t)$, $ t\in [0,T] $, to \eqref{original_system} and \eqref{reduced_system} can be expressed using their fundamental matrices $\Phi(t)$ and $\bar{\Phi}(t)$, respectively, see \cite{redmannspa2}. Therefore, we have
\begin{align*}
x(t;0,u)=\int_0^t\Phi(t,s)Bu(s)\ds,\quad 
\bar{x}(t;0,u)=\int_0^t\bar{\Phi}(t,s)B_1u(s)\ds,
\end{align*}
	where $ \Phi(t,s)=\Phi(t)\Phi^{-1}(s) $ and $ \bar{\Phi}(t,s)=\bar{\Phi}(t)\bar{\Phi}^{-1}(s) $. Consequently, representations for the outputs are
	\begin{equation}
\begin{aligned}
			y(t)&=C x(t;0,u)=C\int_0^t\Phi(t,s)Bu(s)\ds,\\
			\bar{y}(t)&=C_1\bar x(t;0,u)=C_1\int_0^t\bar{\Phi}(t,s)B_1u(s)\ds,
\end{aligned}
	\end{equation}
	where $ t\in[0,T] $. Then, we find
	\begin{equation}
\begin{aligned}
			\mathbb{E}\|y(t)-\bar{y}(t)\|_2&=\mathbb{E}\Big\|C\int_0^t\Phi(t,s)Bu(s)\ds-C_1\int_0^t\bar{\Phi}(t,s)B_1u(s)\ds\Big\|_2\\
			&\leq\mathbb{E}\int_0^t\Big\|\left(C\Phi(t,s)B-C_1\bar{\Phi}(t,s)B_1\right)u(s)\Big{\|}_2\ds\\
			&\leq \mathbb{E}\int_0^t\Big\|C\Phi(t,s)B-C_1\bar{\Phi}(t,s)B_1\Big{\|}_F\|u(s)\|_2\ds.
\end{aligned}
	\end{equation}
	Here, $ \|\cdot\|_F $ denotes the Frobenius norm. Using Cauchy’s inequality, it holds that
	\begin{align*}
		\mathbb{E}\|y(t)-\bar{y}(t)\|_2&\leq \left(\mathbb{E}\int_0^t\Big\|C\Phi(t,s)B-C_1\bar{\Phi}(t,s)B_1\Big{\|}_F^2\ds\right)^{\frac{1}{2}}\left(\mathbb{E}\int_0^t\|u(s)\|_2^2 \ds\right)^{\frac{1}{2}}\\
		&=\left(\mathbb{E}\int_0^t\Big\|C^e \Phi^e(t, s) B^e\Big\|_F^2 \ds\right)^{\frac{1}{2}}\left(\mathbb{E}\int_0^t\|u(s)\|_2^2 \ds\right)^{\frac{1}{2}},
	\end{align*}
where $B^e=\smat B\\ B_1\srix$, $C^e = \smat C & -C\srix$ and $\Phi^e= \smat\Phi & 0\\0 &\bar{\Phi}\srix$ is the fundamental solution to the system with coefficients $A^e=\smat A& 0\\0 &A_{11}\srix $ and $N_i^e=\smat N_i& 0\\0 & N_{i, 11}\srix$.
	
Applying the arguments that are used in \cite{redmannbenner,redmannspa2}, we know that 
	\begin{equation}\label{mean_semi_group}
			\mathbb{E}[\Phi^e(t,s)B^e{B^e}^\top{\Phi^e}^\top(t,s)]=\mathbb{E}[\Phi^e(t-s)B^e{B^e}^\top{\Phi^e}^\top(t-s)].
	\end{equation}	
	For $t\in [0, T]$, the identity in \eqref{mean_semi_group} yields
	\begin{equation}\label{first_bound}
\begin{aligned}
			\mathbb{E}\int_0^t\Big\|C^e \Phi^e(t, s) B^e\Big\|_F^2 \ds
			&=\mathbb{E}\int_0^t \trace(C^e \Phi^e(t, s) B^e{B^e}^\top{\Phi^e}^\top(t,s){C^e}^\top)\ds\\
			&= \mathbb{E}\int_0^t \trace(C^e \Phi^e(s) B^e{B^e}^\top{\Phi^e}^\top(s){C^e}^\top)\ds\leq \trace\Big(C^e \int_0^T F^e(s) \ds\, {C^e}^\top\Big)
\end{aligned}
	\end{equation}
	with $F^e(t) = \mathbb{E}\left[\Phi^e(t) B^e{B^e}^\top{\Phi^e}^\top(t)\right]$ exploiting Fubini's theorem as well as the fact that the trace and $C^e$ are linear operators.  Since $F(t) = \mathbb{E}\left[\Phi(t) BB^\top\Phi^\top(t)\right]$ is a stochastic representation for equation \eqref{matrixequalforF}, see Section \ref{sec:gram}, $F^e$ satisfies \begin{align}\label{matrixequalforF_error}
		{\dot F}^e(t) = A^eF^e(t)+F^e(t){A^e}^\top+ \sum_{i, j=1}^q N_i^e F^e(t) {N_j^e}^\top k_{ij},\quad F^e(0)=B^e{B^e}^\top,
	\end{align}
using the same arguments. From \eqref{matrixequalforF_error}, it can be seen that the left upper $n\times n$ block of $F^e$ is $F$ which solves \eqref{matrixequalforF}. On the other hand, the right lower  $r\times r$ block $\bar F$ and the right upper $n\times r$ block $\tilde F$ of $F^e$ satisfy
\begin{align}\label{odebar}
 \dot{\bar{F}}(t)&=A_{11}\bar{F}(t)+\bar{F}(t)A_{11}^\top+ \sum_{i, j=1}^q N_{i,11} \bar{F}(t) N_{j,11}^\top k_{ij},\quad \bar{F}(0)=B_1B_1^\top,\\ \label{odetilde}
 \dot{\tilde{F}}(t)&=A\tilde{F}(t)+\tilde{F}(t)A_{11}^\top+ \sum_{i, j=1}^q N_{i} \tilde{F}(t) N_{j,11}^\top k_{ij},\quad \tilde{F}(0)=BB_1^\top,
\end{align}
with stochastic representations \begin{align}\label{stocha_rep}
                        \bar{F}(t)=\mathbb{E}[\bar{\Phi}(t)B_1B_1^\top\bar{\Phi}^\top(t)], \quad \tilde{F}(t)=\mathbb{E}[\Phi(t)BB_1^\top\bar{\Phi}^\top(t)].        
                                \end{align}
Consequently, using \eqref{first_bound} with the partition $F^e= \smat F & \tilde F\\\tilde F^\top & \bar F \srix$, we find \begin{align*}
			\mathbb{E}\int_0^t\Big\|C^e \Phi^e(t, s) B^e\Big\|_F^2 \ds
			\leq \trace\Big(C P_T\, {C}^\top\Big) + \trace\Big(C_1 \bar P_T C_1^\top\Big)-2 \trace\Big(C \tilde{P}_T C_1^\top\Big),
		\end{align*}
where $\bar{P}_T=\int_0^T\bar{F}(t)\dt$ and $\tilde{P}_T=\int_0^T\tilde{F}(t)\dt$	
solve 
	\begin{align}\label{barPT}
		\bar{F}(T)-B_1B_1^\top &= A_{11}\bar{P}_T+\bar{P}_T A_{11}^\top+\sum_{i, j=1}^q N_{i,11} \bar{P}_T N_{j,11}^\top k_{ij},\\ \label{tildePT}
		\tilde{F}(T)-BB_1^\top &= A\tilde{P}_T+\tilde{P}_T A_{11}^\top+\sum_{i, j=1}^q N_{i} \tilde{P}_T N_{j,11}^\top k_{ij}.	 	
	\end{align}
Summing up, we obtain that  
	\begin{equation}\label{error_1}
			\sup_{t\in[0,T]}\mathbb{E}\|y(t)-\bar{y}(t)\|_2\leq \left(\trace(CP_TC^\top)+\trace(C_1\bar{P}_TC_1^\top)-2\trace(C\tilde{P}_TC_1^\top) \right)^{\frac{1}{2}}	 \|u\|_{L_T^2}.
	\end{equation}
The bound in \eqref{error_1} is very useful in order to check for the quality of a reduced system. Since $P_T$ has to be computed to obtain \eqref{reduced_system}, the actual cost to determine the bound lies in solving the low-dimensional matrix equations \eqref{barPT} and \eqref{tildePT}. However, \eqref{error_1} is only an a-posteriori estimate which is computed after the reduced order model is derived. Therefore, we discuss the role of $\Sigma_{T, 2}=\diag(\sigma_{T, r+1}, \ldots, \sigma_{T, n})$ which is either the matrix of neglected eigenvalues of $P_T$ or HSVs of the system. $\Sigma_{T, 2}$ is associated to the truncated state variables $x_2$ of \eqref{stochstate_trans}, compare with \eqref{partitions}. By  \eqref{diffreachjanein} and \eqref{outenergiewithremainder}, it is already known that such variables $x_2$ are less relevant if $\sigma_{T, r+1}, \ldots, \sigma_{T, n}$ are small. This makes the values $\sigma_i$ a good a-priori criterion for the choice of $r$. In the following, we want to investigate how the truncated values $\sigma_{T, r+1}, \ldots, \sigma_{T, n}$ characterize the error of the approximation. For that reason, we prove an error bound depending on $\Sigma_{T, 2}$. As we will see, $\Sigma_{T, 2}$ is not the only factor having an impact on the bound that is structurally independent of whether we choose $S=S_1$ or $S= S_2$.
	
	\begin{thm}\label{thm_error_bound}
Let $y$ be the output of \eqref{original_system} and $\bar y$ be the one of \eqref{reduced_system}. Suppose that $S=S_1, S_2$, where $S_1$ is the factor of the eigenvalue decomposition of the Gramian $P_T$ and $S_2$ is the balancing transformation defined in \eqref{balancedtrans}. Using partition \eqref{partitions} of the realization $(A_S, B_S, C_S, N_{i, S})$, we have
		\begin{align*}
		\sup_{t\in[0,T]}\mathbb{E}\|y(t)-\bar{y}(t)\|_2	
				&\leq \Bigg(\trace\bigg( \Sigma_{T,2} \bigg[ C_2^\top C_2+2A_{12}^\top \tilde{Q}_2+\sum_{i,j=1}^q N_{i,12}^\top \Big(2\tilde{Q}\begin{pmatrix} {N}_{j,12} \\ {N}_{j,22} \end{pmatrix}-\bar{Q}N_{j,12}\Big) k_{ij}\bigg] \bigg)\\
				&\quad\quad+2\trace\bigg( \tilde{Q}\begin{pmatrix} \tilde{F}_1-{F}_{11} \\ \tilde{F}_2-{F}_{21} \end{pmatrix}\bigg)+\trace\bigg(\bar{Q}(F_{11}-\bar{F}) \bigg)\Bigg)^{\frac{1}{2}} \|u\|_{L_T^2},
		\end{align*}
where $\bar{Q}$ and $\tilde{Q}=\begin{pmatrix}
			\tilde{Q}_1\quad  \tilde{Q}_2  
		\end{pmatrix}$ and are the unique solutions to 
		\begin{align}\label{barQ}
		A_{11}^\top \bar{Q}+\bar{Q} A_{11}+\sum_{i,j=1}^q N_{i,11}^\top\bar{Q}N_{j,11}k_{ij}&=-C_1^\top C_1,\\
		\label{tildeQ}
			A_{11}^\top \tilde{Q}+\tilde{Q} A_S+\sum_{i,j=1}^q N_{i,11}^\top\tilde{Q}N_{j, S}k_{ij}&=-C_1^\top C_S.
		\end{align}
Moreover, the above bound involves $F_S(T):=SF(T)S^{\top}=\smat 
			F_{11} & 	F_{12} \\
			F_{21} & 	F_{22} 
		\srix$ and $\tilde{F}_S(T):= S\tilde{F}(T)=\smat 
		{\tilde{F}_1} \\
		{\tilde{F}_2 } 
	\srix$, where $F(T)$, $\bar F=\bar F(T)$ and $\tilde F(T)$ are the terminal values of \eqref{matrixequalforF}, \eqref{odebar} and \eqref{odetilde}, respectively.
	\end{thm}
The terms in the bound of Theorem \ref{thm_error_bound} that do not directly depend on $\Sigma_{T, 2}$ are related to the covariance error of the dimension reduction at the terminal time $T$ (with $u\equiv 0$). 
To see this, let $V$ be the matrix introduced in \eqref{defineV}. As explained below \eqref{defineV}, the state of the reduced system \eqref{reduced_system} can be interpreted as an approximation of the original state in the subspace spanned by the columns of $V$. By the stochastic representations of $F(T)$, $\tilde F(T)$ and $\bar F(T)$ (see above \eqref{matrixequalforF} and \eqref{stocha_rep}), we can view $F(T)$ and $\bar F(T)$ as covariances of the original and reduced model at time $T$, whereas $\tilde F(T)$ describes the correlations between both systems. Let us now assume that \begin{align}\label{approx1}
F(T) &\approx \tilde F(T)V^\top,\\ \label{approx2}
F(T) &\approx V  \bar F(T) V^\top,                                                                                                                                                                                                                                                           \end{align}
i.e., the covariance at $T$ is well-approximated in the reduced system. This is, e.g., given if the uncontrolled state is well-approximated in the range of $V$ at time $T$, i.e., $\Phi(T) B \approx V \bar\Phi(T) B_1$. Now, multiplying \eqref{approx1} with $S$ from the left and with $W$ (defined in \eqref{defineV}) from the right, we obtain 
that $\smat {\tilde{F}_1-{F}_{11}} \\ {\tilde{F}_2-{F}_{21} }\srix$ is small. Multiplying \eqref{approx2} with $W^\top$ from the left and with $W$ from the right provides a low deviation between $F_{11}= W^\top F(T) W$ and $\bar{F}$. Although we additionally have these terms related to the covariance error, looking at $\Sigma_{T, 2}$ is still suitable for getting an intuition concerning the error and hence a first idea for the choice of $r$. This is because a small $\Sigma_{T, 2}$ goes along with a small error between $\Phi(T) B$ and its approximation $V \bar\Phi(T) B_1$ in the range of $V$. This observation can be made due to\begin{align*}                                                                                                                                                                                                                                                        \mathbb{E}\int_0^T\|\left(\Phi(t) B\right)^\top z_T \|_2^2\dt = z_T^\top P_T z_T = 0,                                                                                                                                                                                                                                                                                                                                                                                                                                                                                                                                                                                                                                       \end{align*}
 where $z_T\in \kernel P_T$. Since $t\mapsto\Phi(t)$ is $\mathbb P$-almost surely continuous, we have $\left(\Phi(t) B\right)^\top z_T=0$ $\mathbb P$-almost surely for all $t\in [0, T]$. Choosing $t=T$, we therefore know that the columns of $\Phi(T) B$ are orthogonal to $\kernel P_T$. This means that $\Phi(T) B\in \im P_T$ since $P_T$ is symmetric. Hence, there is a matrix $Z_T$ such that \begin{align*}
  \Phi(T) B =  P_T  Z_T  =  S^{-1} \Sigma_T S^{-\top} Z_T = \begin{pmatrix}
			V &  \star 
		\end{pmatrix} \begin{pmatrix}
			\Sigma_{T, 1} &  \\
			& \Sigma_{T,2}
		\end{pmatrix}   \begin{pmatrix}
			V^\top \\  \star 
		\end{pmatrix} Z_T\approx V  \Sigma_{T, 1} V^\top Z_T,                                                                                                                                                                                                                                                                                                                                                                                                                                                                                                                                                                                                                                \end{align*}
i.e., the columns of $ \Phi(T) B$ lie almost in the span of $V$ if $\Sigma_{T,2}$ is small. Therefore, a good approximation can be expected if one truncates states with associated small values $\sigma_{T, r+1},\ldots,\sigma_{T, n}$. This can be confirmed by computing the representation in \eqref{error_1} after a reduced order dimension $r$ was chosen based on the values $\sigma_{T, i}$. 
\begin{remark}
Notice that the the covariance $F(T)$ vanishes in the limit as $T\rightarrow \infty$ if  \eqref{original_system} is mean square asymptotically stable. In this context, the deviations in \eqref{approx1} and \eqref{approx2} can be expected to be small for sufficiently large $T$ since the covariance error disappears at $\infty$. If the system is unstable, we have  $\|F(T)\| \rightarrow \infty$ as $T\rightarrow \infty$. In this case, the covariance error might be large and dominant if $T$ is very large such that the approximation quality is lower. The role of $T$ is additionally discussed in Section \ref{numerics section}.
\end{remark}
We are now ready to prove the error bound in the following:
	\begin{proof}[Proof of Theorem \ref{thm_error_bound}]
		Since $S= S_1, S_2$ diagonalizes $P_T$, we have 
		\begin{equation}\label{balancePT}
			A_S\Sigma_T+\Sigma_TA_S^\top+\sum_{i,j=1}^q N_{i, S} \Sigma_T N_{j, S}^\top k_{ij}=-B_S B_S^\top +F_S(T).
		\end{equation}
	We set	$\tilde{Y}_T:=S\tilde{P}_T$
		and obtain the corresponding equation by multiplying \eqref{tildePT} with $S$ from the left resulting in
		\begin{equation}\label{tildeY}
			A_S\tilde{Y}_T+\tilde{Y}_TA_{11}^\top +\sum_{i,j=1}^q N_{i,S}\tilde{Y}_T N_{j,11}^\top k_{ij}=-B_SB_1^\top+\tilde{F}_S(T).
			\end{equation}
Now, we analyze the trace expression $\epsilon^2 :=\left(\trace(CP_TC^\top)+\trace(C_1\bar{P}_TC_1^\top)-2\trace(C\tilde{P}_TC_1^\top) \right)$ in \eqref{error_1}. We see that
		\begin{equation}\label{epsilon}
\begin{aligned}
				\epsilon^2&=\left(\trace(C_S\Sigma_TC_S^\top)+\trace(C_1\bar{P}_TC_1^\top)-2\trace(C_S\tilde{Y}_TC_1^\top) \right)\\
				&=\left(\trace(C_1\Sigma_{T,1}C_1^\top)+\trace(C_2\Sigma_{T,2}C_2^\top)+\trace(C_1\bar{P}_TC_1^\top)-2\trace(C_S\tilde{Y}_TC_1^\top) \right).
\end{aligned}
		\end{equation}
Exploiting \eqref{tildeQ} yields
	\begin{align*}
		-\trace(C_S\tilde{Y}_TC_1^\top)&= -\trace(\tilde{Y}_TC_1^\top C_S)=\trace\left( \tilde{Y}_T\left[ A_{11}^\top \tilde{Q}+\tilde{Q} A_S+\sum_{i,j=1}^q N_{i,11}^\top\tilde{Q}N_{j, S}k_{ij} \right]\right)\\
			&=\trace\left(\tilde{Q} \left[ A_S\tilde{Y}_T +\tilde{Y}_T A_{11}^\top+\sum_{i,j=1}^q N_{i, S}\tilde{Y}_TN_{j,11}^\top k_{ij} \right]\right).
	\end{align*}
 Comparing \eqref{tildeQ} and  \eqref{tildeY}, we find that
   \begin{equation}\label{chtildeY}
   	-\trace(C_S\tilde{Y}_TC_1^\top)= -\trace(\tilde{Q}B_SB_1^\top)+\trace(\tilde{Q}\tilde{F}_S(T)).
   \end{equation}
Using the partition in \eqref{partitions}, the first $r$ columns of \eqref{balancePT} are 
	\begin{equation}\label{above_eq}
\begin{aligned}
			&\begin{pmatrix} {A}_{11} \\ {A}_{21} \end{pmatrix}\Sigma_{T,1}+\begin{pmatrix} \Sigma_{T,1}{A}_{11}^\top \\ \Sigma_{T,2}{A}_{12}^\top \end{pmatrix}+\sum_{i,j=1}^q \left(\begin{pmatrix} {N}_{i,11} \\ {N}_{i,21} \end{pmatrix}\Sigma_{T,1} N_{j,11}^\top+\begin{pmatrix} {N}_{i,12} \\ {N}_{i,22} \end{pmatrix}\Sigma_{T,2} N_{j,12}^\top  \right)k_{ij}\\
			&=-B_SB_1^\top+\begin{pmatrix} {F}_{11} \\ {F}_{21} \end{pmatrix}.
\end{aligned}
	\end{equation}
	We insert \eqref{above_eq} into \eqref{chtildeY} and obtain
	{\allowdisplaybreaks \begin{align*}
			-\trace(C_S\tilde{Y}_TC_1^\top)&= \trace\left( \tilde{Q}\begin{pmatrix} \tilde{F}_1-{F}_{11} \\ \tilde{F}_2-{F}_{21} \end{pmatrix}\right)\\
			&+\trace \left( \tilde{Q} \left[ \begin{pmatrix} {A}_{11} \\ {A}_{21} \end{pmatrix}\Sigma_{T,1}+\begin{pmatrix} \Sigma_{T,1}{A}_{11}^\top \\ \Sigma_{T,2}{A}_{12}^\top \end{pmatrix}+\sum_{i,j=1}^q \left(\begin{pmatrix} {N}_{i,11} \\ {N}_{i,21} \end{pmatrix}\Sigma_{T,1} N_{j,11}^\top+\begin{pmatrix} {N}_{i,12} \\ {N}_{i,22} \end{pmatrix}\Sigma_{T,2} N_{j,12}^\top  \right)k_{ij}  \right]\right)\\
			&=\trace\left( \tilde{Q}\begin{pmatrix} \tilde{F}_1-{F}_{11} \\ \tilde{F}_2-{F}_{21} \end{pmatrix}\right)+\trace\left( \Sigma_{T,2} \left[ A_{12}^\top \tilde{Q}_2+\sum_{i,j=1}^q N_{i,12}^\top\tilde{Q}\begin{pmatrix} {N}_{j,12} \\ {N}_{j,22} \end{pmatrix}k_{ij}\right] \right)\\
			&\quad+\trace\left( \Sigma_{T,1} \left[ \tilde{Q}\begin{pmatrix} {A}_{11} \\ {A}_{21} \end{pmatrix}+A_{11}^\top \tilde{Q}_1+\sum_{i,j=1}^q N_{i,11}^\top\tilde{Q}\begin{pmatrix} {N}_{j,11} \\ {N}_{j,21} \end{pmatrix}k_{ij}\right] \right).
	\end{align*}}	
Using the partition of the balanced realization in \eqref{partitions}, we observe that the last term of above equation is the first $r$ columns of \eqref{tildeQ}. So, we can say that
	\begin{equation}\label{crossgramtrans}
\begin{aligned}
			-\trace(C_S\tilde{Y}_TC_1^\top)&=\trace\left( \tilde{Q}\begin{pmatrix} \tilde{F}_1-{F}_{11} \\ \tilde{F}_2-{F}_{21} \end{pmatrix}\right)+\trace\left( \Sigma_{T,2} \left[ A_{12}^\top \tilde{Q}_2+\sum_{i,j=1}^q N_{i,12}^\top\tilde{Q}\begin{pmatrix} {N}_{j,12} \\ {N}_{j,22} \end{pmatrix}k_{ij}\right] \right)\\
			&\quad -\trace(\Sigma_{T,1} C_1^\top C_1).
\end{aligned}
	\end{equation}
	Inserting \eqref{crossgramtrans} into \eqref{epsilon}, we have
	\begin{equation}
\begin{aligned}
			\epsilon^2&= \trace\left(  \Sigma_{T,2} \left[C_2^\top C_2+2A_{12}^\top \tilde{Q}_2+2 \sum_{i,j=1}^q N_{i,12}^\top\tilde{Q}\begin{pmatrix} {N}_{j,12} \\ {N}_{j,22} \end{pmatrix}k_{ij}\right] \right)\\
			&\quad + 2\trace\left( \tilde{Q}\begin{pmatrix} \tilde{F}_1-{F}_{11} \\ \tilde{F}_2-{F}_{21} \end{pmatrix}\right)+\trace\left( (\bar{P}_T-\Sigma_{T,1})C_1^\top C_1\right).
\end{aligned}
\end{equation}	

Equation \eqref{barQ} now yields
	\begin{align*}
			&\trace\left( (\bar{P}_T-\Sigma_{T,1})C_1^\top C_1\right)\\
			&=-\trace\left( \bar{Q}\left[ A_{11}(\bar{P}_T-\Sigma_{T,1})+(\bar{P}_T-\Sigma_{T,1}) A_{11}^\top+\sum_{i,j=1}^q N_{i,11}(\bar{P}_T-\Sigma_{T,1})N_{j,11}^\top k_{ij}\right)\right]
	\end{align*} 	
The combination of \eqref{barPT} and the left upper block of \eqref{balancePT} gives
	\begin{align*}
			 &A_{11}(\bar{P}_T-\Sigma_{T,1})+(\bar{P}_T-\Sigma_{T,1}) A_{11}^\top+\sum_{i,j=1}^q N_{i,11}(\bar{P}_T-\Sigma_{T,1})N_{j,11}^\top k_{ij}\\
			 & =\sum_{i,j=1}^q N_{i,12}\Sigma_{T,2}N_{j,12}^\top k_{ij}+(\bar{F}-F_{11}).
	\end{align*}	
		Consequently, we have
		\[\trace\left( (\bar{P}_T-\Sigma_{T,1})C_1^\top C_1\right)=-\trace\left( \Sigma_{T,2}\left[ \sum_{i,j=1}^q N_{i,12}^\top \bar{Q}N_{j,12} k_{ij}\right]\right)+\trace\left(\bar{Q}(F_{11}-\bar{F})\right).\]
		So, we obtain that
		\begin{align*}
				\epsilon^2&= \trace\left( \Sigma_{T,2} \left[ C_2^\top C_2+2A_{12}^\top \tilde{Q}_2+\sum_{i,j=1}^q N_{i,12}^\top \Big(2\tilde{Q}\begin{pmatrix} {N}_{j,12} \\ {N}_{j,22} \end{pmatrix}-\bar{Q}N_{j,12}\Big) k_{ij}\right] \right)\\
				&\quad+2\trace\left( \tilde{Q}\begin{pmatrix} \tilde{F}_1-{F}_{11} \\ \tilde{F}_2-{F}_{21} \end{pmatrix}\right)+\trace\left(\bar{Q}(F_{11}-\bar{F})\right),
					\end{align*}
which concludes the proof of this theorem.
	\end{proof}
Notice that the estimate in Theorem \ref{thm_error_bound} is also beneficial if $N_i=0$ for all $i=1, \ldots, q$, since it improves the deterministic bound \cite{redmannkuerschner} in the sense that we can generally deduce the relation between the truncated HSVs and the actual approximation error here. It is important to notice that, in the deterministic case, ``improvement'' is not meant in terms of accuracy. The error bound representation in \cite{redmannkuerschner} just has the drawback that it allows to make similar conclusions only if the underlying system is asymptotically stable. 
Moreover, the result of Theorem \ref{thm_error_bound} is a generalization of the bounds for mean square asymptotically stable stochastic systems \cite{redmannbenner, redmannigor}, where the covariance related terms vanish as $T\rightarrow \infty$. 

	\section{Computation of Gramians}\label{sec_compGram}
In this section, we discuss how to compute $P_T$ and $Q_T$ which allow us to identify redundant information in the system. These matrices are solutions of Lyapunov equations \eqref{comPT} and \eqref{gengenlyapobs} with left hand sides depending on $F(T)$ and $G(T)$, respectively. Given $F(T)$ and $G(T)$ it is therefore required to solve generalized Lyapunov equations \begin{align}\label{example_eq}
     L= \mathcal L_A(X) + \Pi(X)                                                                                                                                                                                                                                                                                                                                                                                                                                                                                                                                                                                                                        \end{align}
efficiently, where $L$ is a symmetric matrix of suitable dimension. According to Remark \ref{rem_vect} this can be done by vectorization, i.e., one can try to solve $\vect\left(L\right)= \mathcal K \vect(X)$ with the Kronecker matrix $\mathcal K$ defined in \eqref{stochstab}. Since $\mathcal K$ is of order $n^2$, the complexity of deriving $\vect(X)$ from this linear system of equations is $\mathcal{O}(n^6)$ making this procedure infeasible for $n\gg 100$.

However, more efficient techniques have been developed in order to solve \eqref{example_eq}, see, e.g.,  \cite{damm_lyap_eq}, where a sequence of standard Lyapunov equations ($\Pi =0$) is solved to find $X$. Such standard Lyapunov equations can either
be tackled by direct methods, such as Bartels-Stewart \cite{Bartels_Stewart}, which cost $\mathcal{O}(n^3)$ operations, or by iterative methods such as ADI or Krylov subspace methods
\cite{Simoncini_ueberblick}, which have a much  smaller complexity than the Bartels-Stewart algorithm, in particular, when the left hand side is of low rank or structured (complexity of $\mathcal{O}(n^2)$ or less). 

Solving for $P_T$ and $Q_T$ now relies on having access to $F(T)$ and $G(T)$ which are the terminal values of the matrix-differential equations \eqref{matrixequalforF} and \eqref{ODEforG}. The remainder of this section will deal with strategies to compute these terminal values.

\subsection{Exact methods}\label{exact_Gram}

 One solution to overcome the issue of unknown $F(T)$ and $G(T)$ is to use vectorizations of \eqref{matrixequalforF} and \eqref{ODEforG} for dimensions $n$ of a few hundreds. If we define $f(t):=\vect(F(t))$ and $g(t)=\vect(G(t))$, then
\begin{align*}
		\dot f(t)=\mathcal K f(t), \quad f(0)=\vect(BB^\top),\quad \dot g(t)=\mathcal K^\top g(t), \quad g(0)=\vect(C^\top C),
\end{align*}  
where $ \mathcal K$ is defined in \eqref{stochstab}. Therefore, obtaining $F(T)$ and $G(T)$ rely on the efficient computation of a matrix exponential, since
\begin{align*}
f(T)=\expn^{\mathcal K T}\vect(BB^\top), \quad g(T)=\expn^{\mathcal K^\top T}\vect(C^\top C).
\end{align*}
One can find a discussion on how to determine a matrix exponential efficiently in \cite{kurschner2018balanced} and references therein. Alternatively, one might think of discretizing the matrix differential equations \eqref{matrixequalforF} and \eqref{ODEforG} to find an approximation of $F(T)$ and $G(T)$. However, as stated above, these equations are equivalent to ordinary differential equations of order $n^2$. Solving such extremely large scale systems is usually not feasible. In addition, only implicit schemes would allow for a reasonable step size in the discretization making the problem even more complex. For that reason, we discuss more suitable numerical approximations in the following.  

\subsection{Sampling based approaches}\label{sampled_Gram}

We aim to derive an approximation of the terminal value $F(T) = \mathbb E [\Phi(T)BB^\top \Phi^\top(T)]$ of \eqref{matrixequalforF} by different stochastic representations. This alternative approach is required since computing $\expn^{\mathcal K T}$ is not feasible if $n\gg 100$ knowing that $\mathcal K\in \mathbb R^{n^2\times n^2}$. Therefore, we discuss sampling based approaches in the following. Let $\Phi^i(T)$, $i\in\{1, \ldots, M\}$, be i.i.d. copies of $\Phi(T)$. Then, we have $\frac{1}{M}\sum_{i=1}^M \Phi^i(T)BB^\top {\Phi^i}(T)^\top\approx F(T)$ if $M$ is sufficiently large. This requires to sample the random variable $\Phi(T)B$ possibly many times.  $\Phi(T)B$ is the terminal value of the stochastic differential equation \begin{align}\label{generatesampleseq}
	dx_B(t) = Ax_B(t) dt+\sum_{i=1}^q N_i x_B(t) dw_i(t),\quad x_B(0)=B,                                                                                                                                                                                                                                                                                                                                                                                                                                                                                                                                                                                                                                                                                                     \end{align}
with $x_B(t) \in \mathbb R^{n\times m}$. System \eqref{generatesampleseq} can be seen as a matrix-valued homogeneous version of \eqref{stochstatenew} ($u\equiv 0$) with initial state $B$. If \eqref{original_system} needs to be evaluated for many different controls $u$ and additionally a large number of samples are required for each fixed $u$, it even pays off to generate many samples of the solution to \eqref{generatesampleseq}. In particular, this is true if the number of columns of $B$ is low. However, we want to avoid evaluating \eqref{generatesampleseq} too often. The number of samples $M$ required for a good estimate of $F(T)$ depends on the variance of $\Phi(T)BB^\top \Phi^\top(T)$. Therefore, we want to reduce the variance by finding a better stochastic representation than $\mathbb E [\Phi(T)BB^\top \Phi^\top(T)]$. In the spirit of variance reduction techniques, we find the zero variance unbiased estimator first. To do so, we apply Ito's product rule (see e.g. \cite{oksendal}) to obtain \begin{align*} 	                                                                                                                                                                                                                                                                                                                                                                                                                                                                                                                                      
  d\left(x_B(t) x_B^\top(t)\right) &=   d\left(x_B(t)\right) x_B^\top(t)  + x_B(t) d\left(x_B^\top(t)\right) +   d\left(x_B(t) \right)d\left(x_B^\top(t)\right)\\
  &=\left(A x_B(t) dt+\sum_{i=1}^q N_i x_B(t) dw_i(t)\right) x_B^\top(t)  + x_B(t) \left(x_B^\top(t) A^\top dt   
  +\sum_{i=1}^q x_B^\top(t)N_i^\top dw_i(t)\right) \\
  &\quad+  \sum_{i,j=1}^q N_i x_B(t) x_B^\top(t) N_{j}^\top k_{ij}dt\\ &=(\mathcal L_A + \Pi)\left(x_B(t) x_B^\top(t)\right)dt+ \sum_{i=1}^q \mathcal L_{N_i}\left(x_B(t)x_B^\top(t)\right) dw_i(t).
  \end{align*}
This stochastic differential is now exploited to find \begin{align*}
d\left(\expn^{\mathcal K (T-t)} \vect(x_B(t) x_B^\top(t))\right) 
&= - \expn^{\mathcal K (T-t)} \mathcal K \vect(x_B(t) x_B^\top(t)) dt + \expn^{\mathcal K (T-t)}  d\left(\vect(x_B(t) x_B^\top(t))\right)\\
&=\sum_{i=1}^q \expn^{\mathcal K (T-t)} \vect\left(\mathcal L_{N_i}\left(x_B(t)x_B^\top(t)\right)\right) dw_i(t)
                                                     \end{align*}
 using that $\vect\left((\mathcal L_A + \Pi)\left(x_B(t) x_B^\top(t)\right)\right)= \mathcal K \vect(x_B(t) x_B^\top(t))$. Hence, we have \begin{align*}
\vect\left(x_B(T) x_B^\top(T)\right)
&= \expn^{\mathcal K T} \vect(BB^\top) + \sum_{i=1}^q \int_0^T \expn^{\mathcal K (T-t)} \vect\left(\mathcal L_{N_i}\left(x_B(t)x_B^\top(t)\right)\right) dw_i(t).
                                                     \end{align*}                                          
Devectorizing this equation yields \begin{align}\label{zerovariancerep}
F(T) = x_B(T) x_B^\top(T) - \sum_{i=1}^q \int_0^T F\left( T-t,  \mathcal L_{N_i}\left(x_B(t)x_B^\top(t)\right)\right) dw_i(t),
                                   \end{align}
where the second argument in $F$ represents the initial condition of \eqref{matrixequalforF}. The right hand side of \eqref{zerovariancerep} now is unbiased zero variance estimator of $F(T)$. However, this estimator depends on $F$ which is not available. Therefore, given a symmetric matrix $X_0$, we approximate $F(t, X_0)$ by a computable matrix function  $\mathcal F(t, X_0)$ that we specify later.
This leads to the unbiased estimator \begin{align}\label{simulated_est_F}
                                                                                                                                                                                                                                        E_{\mathcal F}(T):= x_B(T) x_B^\top(T) - \sum_{i=1}^q \int_0^T \mathcal F\left(T-t, \mathcal L_{N_i}\left(x_B(t)x_B^\top(t)\right) \right) dw_i(t)
                                                                                                                                                                                                                                                    \end{align}                                                                                                                                                                                                                                        for $F(T)$. The hope is that a few samples of $E_{\mathcal F}(T)$ can give an accurate approximation of $F(T)$. Of course, $E_{\mathcal F}(T)$ can only be simulated by further discretizing the above Ito integrals, e.g., by a Riemann-Stieltjes sum approximation. The variance of $E_{\mathcal F}(T)$ is \begin{align*}
      &\mathbb E\Big\|E_{\mathcal F}(T)-F(T)\Big{\|}_F^2 =                                                                                                                                                                                                                                              
  \mathbb E\Big\|\sum_{i=1}^q \int_0^T F\left( T-t, X_i(t)\right) -\mathcal F\left(T-t, X_i(t)\right) dw_i(t)  \Big{\|}_F ^2    \\
  &= \sum_{i, j=1}^q \mathbb E\int_0^T\Big\langle F\left( T-t, X_i(t)\right) -\mathcal F\left(T-t, X_i(t)\right) , F\left( T-t, X_j(t)\right) -\mathcal F\left( T-t, X_j(t)\right) \Big\rangle_F k_{ij} dt  
  \end{align*}
  setting $X_i(t)=  N_i x_B(t)x_B^\top(t) + x_B(t) x_B^\top(t)N_i^\top$  and exploiting Ito's isometry, see \cite{oksendal}. Consequently, the benefit of the variance reduction depends on the difference $F(t, X_0)-\mathcal F(t, X_0)$.   \smallskip
  
  We conclude this section by discussing suitable approximations $\mathcal F(t, X_0)$ of $F(t, X_0)$. For that reason, we establish the following theorem.
  \begin{thm}\label{thm_apprF}
Let $F(t, X_0)$, $t\in [0, T]$, be the solution to \begin{align*}
		{\dot F}(t) = \mathcal L_A\left(F(t)\right)+ \Pi\left(F(t)\right),\quad F(0)=X_0,
	\end{align*}
where the initial data $X_0$ is a symmetric matrix. Then, there exist constants $\underline c$ and $\overline c$ such that \begin{align*}
\expn^{A t} X_0 \expn^{A^\top t} + \underline c\int_0^t \expn^{A s} \Pi\left(I\right) \expn^{A^\top s}  ds \leq F(t) \leq \expn^{A t} X_0 \expn^{A^\top t} + \overline c \int_0^t \expn^{A s} \Pi\left(I\right) \expn^{A^\top s} ds.
	\end{align*}
  \end{thm}
\begin{proof}
Exploiting the product rule, it can be seen that $F$ is implicitly given by \begin{align}\label{mild_sol}
F(t) = \expn^{A t} X_0 \expn^{A^\top t} + \int_0^t \expn^{A (t-s)} \Pi\left(F(s)\right) \expn^{A^\top (t-s)} ds.
	\end{align}
The solution $t\mapsto F(t)$ is continuous and $F(t)$ is a symmetric matrix for all $t\in [0, T]$. Consequently, exploiting \cite[Corollary VI.1.6]{bhatia97}, there exist continuous and real functions $\lambda_1, \ldots, \lambda_n$ such that $\lambda_1(t), \ldots, \lambda_n(t)$ represent the eigenvalues of $F(t)$ for each fixed $t$. We now define continuous functions by $\underline{\lambda} := \min\{\lambda_1, \ldots, \lambda_n\}$ and $\overline{\lambda} := \max\{\lambda_1, \ldots, \lambda_n\}$. Symmetric matrices can be estimated from below and above by their smallest and largest eigenvalue, respectively, leading to $\underline{\lambda}(t)I\leq F(t)\leq \overline{\lambda}(t)I$. Therefore, given an arbitrary vector in $v\in\mathbb R^{n}$, we have {\allowdisplaybreaks \begin{align*}
  v^\top\Pi\left(F(t)\right)v&=\sum_{i,j=1}^q (N_i v)^\top F(t) N_j v k_{ij}=\sum_{i,j=1}^q (N_i v)^\top F(t) N_j v e_i^\top \mathbf K^{\frac{1}{2}} \mathbf K^{\frac{1}{2}} e_j\\&= \sum_{i,j=1}^q 
(N_i v)^\top F(t) N_j v \sum_{k=1}^q \langle \mathbf K^{\frac{1}{2}} e_i, e_k\rangle_2 \langle \mathbf K^{\frac{1}{2}}e_j , e_k\rangle_2\\&= \sum_{k=1}^q 
\Bigg(\sum_{i=1}^q N_i v \langle \mathbf K^{\frac{1}{2}} e_i, e_k\rangle_2\Bigg)^\top F(t) \Bigg(\underbrace{\sum_{j=1}^q N_j v  \langle \mathbf K^{\frac{1}{2}}e_j , e_k\rangle_2}_{=:v_k}\Bigg)\begin{cases}
  \leq \overline{\lambda}(t) \sum_{k=1}^q v_k ^\top  I v_k\\
  \geq   \underline{\lambda}(t)\sum_{k=1}^q v_k^\top  I v_k                                                                                                                                                               \end{cases}
  \end{align*}}
resulting in $\underline{\lambda}(t)\Pi\left(I\right)\leq \Pi\left(F(t)\right)\leq \overline{\lambda}(t)\Pi\left(I\right)$, where $e_i$ is the canonical basis of $\mathbb R^q$.
Since $\underline{\lambda}, \overline{\lambda}$ are continuous on $[0, T]$, they can be bounded from below and above by some suitable constants. Applying this to \eqref{mild_sol}, we obtain the result by substitution.
\end{proof}
Of course, the constants in Theorem \ref{thm_apprF} are generally unknown. However, this result gives us the intuition that $F(t, X_0)$ can be approximated by \begin{align}\label{estimationF}
                                                                                                                                                                \mathcal F(t, X_0)=\expn^{A t} X_0 \expn^{A^\top t} + c \int_0^t \expn^{A s} \Pi\left(I\right) \expn^{A^\top s} ds,
                                                                                                                                                               \end{align}
where $c\in[\underline c, \overline c]$ is a  real number. From the proof of Theorem \ref{thm_apprF}, we further know that $\underline c, \overline c\geq 0$ if $X_0$ is positive semidefinite. We cannot generally expect a reduction of the variance  for all choices of $c$. However, a good candidate will reduce the computational complexity. A general strategy how to find such a candidate is an interesting question for future research.
\begin{remark}\label{rem_complex}
 Besides generating (a few) samples of $x_B$ from \eqref{generatesampleseq}, we require the matrix exponentials $\expn^{A t_i}$ on a grid $0=t_0 < t_1 < \dots < t_{n_g}=T$ to determine the estimator \eqref{simulated_est_F} with $\mathcal F$ as in \eqref{estimationF}. Here, $n_g$ is the number of grid points when discretizing the Ito integral in \eqref{simulated_est_F}. If the points $t_i$ are equidistant with step size $h$, one first  computes $\expn^{A h}$. The other exponentials are then powers of $\expn^{A h}$ such that a certain number of matrix multiplications (depending on $n_g$) have to be conducted.
\end{remark}
The Gramian $Q_T$ can be computed from \eqref{gengenlyapobs} requiring to determine  $G(T)$. 
According to Remark \ref{remcomQ}, we know that $G(T) = \mathbb E [x_C(T) x_C^\top(T)]$, where
\begin{align*}
	dx_C(t) = A^\top x_C(t) dt+\sum_{i=1}^q N_i^\top x_C(t) dw_i(t),\quad x_C(0)=C^\top,                                                                                                                                                                                                                                                                                                                                                                                                                                                                                                                                                                                                                                                                                                     \end{align*}
with $x_C(t)\in \mathbb R^{n\times p}$. Exploiting the above consideration regarding $F(T)$, we can see that \begin{align}\label{simulated_est_G}                                                                                                                                                                                                                            E_{\mathcal G}(T):= x_C(T) x_C^\top(T) - \sum_{i=1}^q \int_0^T \mathcal G\left(T-t, \mathcal L_{N_i}^*\left(x_C(t)x_C^\top(t)\right) \right) dw_i(t)
                                                                                                                                                                                                                                                    \end{align}  
is a possible unbiased estimator for $G(T)$. The approximation $\mathcal G$ of $G$ can be chosen as in \eqref{estimationF} replacing $(A, N_i)\mapsto (A^\top, N_i^\top)$.

\subsection{Gramians based on deterministic approximations of $F(T)$ and $G(T)$}\label{approxFG_Gram}

Based on Theorem \ref{thm_apprF}, an estimation of $F(T)$ (and also $G(T)$) is given in \eqref{estimationF}. Instead of using these approximations in a variance reduction procedure like in Section \ref{sampled_Gram}, we exploit it directly in \eqref{comPT} and \eqref{gengenlyapobs}. This leads to matrices $\mathcal P_T$ and $\mathcal Q_T$ solving \begin{align}
		\mathcal F(T, BB^\top)-BB^\top &= 
		\mathcal L_A\left(\mathcal P_T\right) + \Pi\left(\mathcal P_T\right),\\ 
    \mathcal G(T, C^\top C)-C^\top C &= 
		\mathcal L_A^*\left(\mathcal Q_T\right) + \Pi^*\left(\mathcal Q_T\right),
	\end{align}
 where the left hand sides are defined by 
 \begin{align}     \label{calP}                                                                                                                                                           \mathcal F(T, BB^\top)&=\expn^{A T} BB^\top\expn^{A^\top T} + c_F \int_0^T \expn^{A s} \Pi\left(I\right) \expn^{A^\top s} ds,\quad c_F\in\mathbb R,\\ \label{calQ}
\mathcal G(T, C^\top C)&=\expn^{A^\top T} C^\top C \expn^{A T} + c_G \int_0^T \expn^{A^\top s} \Pi^*\left(I\right) \expn^{A s} ds, \quad c_G\in\mathbb R.                                                                                                                                                              \end{align}
Certainly, the choice of the constants $c_F$ and $c_G$ determine how well $P_T$ and $Q_T$ are approximated by $\mathcal P_T$ and $\mathcal Q_T$, e.g., in terms of the characterization of the respective dominant subspaces of system \eqref{original_system}. Notice that for $N_i=0$, $\mathcal F(T, BB^\top)$ and $\mathcal G(T, C^\top C)$ yield the exact values for $F(T, BB^\top)$ and $G(T, C^\top C)$. At this point, it is important to mention 
that the Gramian approximation of this section is computationally less complex than the one in Section \ref{sampled_Gram}. First of all, we do not need to sample from \eqref{generatesampleseq} and secondly no Ito integral as in \eqref{simulated_est_F} has to be discretized. Calculating $\mathcal F$ and $\mathcal G$ might also require to compute matrix exponentials on a partition of $[0, T]$, compare with Remark \ref{rem_complex}. However, less grid points than for the sampled Gramians of Section \ref{sampled_Gram} have to be considered since an ordinary integral can be discretized with a larger step size compared to an Ito integral. Alternatively, the integrals in \eqref{calP} and \eqref{calQ} can also be determined without a discretization since it holds that \begin{align*}
\mathcal L_A\bigg(\int_0^T \expn^{A s} \Pi\left(I\right) \expn^{A^\top s} ds\bigg) &= - \Pi\left(I\right)  +\expn^{A T} \Pi\left(I\right) \expn^{A^\top T},\\
\mathcal L_A^*\bigg(\int_0^T \expn^{A^\top s} \Pi^*\left(I\right) \expn^{A s} ds\bigg) &= - \Pi^*\left(I\right)  +\expn^{A^\top T} \Pi^*\left(I\right) \expn^{A T}.                                                                                                                                                                                                                                                                                                                                                                                                                                                                                                                                                                                                                                                                                                                                                                                                                                                                                                                                                                                                                                                     \end{align*}
This approach has the advantage that only the matrix exponential $\expn^{A T}$ at the terminal time is needed.

\section{Numerical experiments}\label{numerics section}
In order to indicate the benefit of the model reduction method presented in Section \ref{sec_mor}, we consider a linear controlled SPDE as in \eqref{heateq}. In addition, we emphasize the applicability to unstable systems by rescaling and shifting the Laplacian. The concrete example of interest is \begin{align*}
			\frac{\partial {\mathcal X}(t,\zeta)}{\partial t}&=\left(\alpha\Delta+\beta I\right) {\mathcal X}(t,\zeta)+1_{[\frac{\pi}{4},\frac{3\pi}{4}]^2}(\zeta)u(t)+\gamma\expn^{-|\zeta_1-\frac{\pi}{2}|-\zeta_2}{\mathcal X}(t,\zeta)\frac{\partial w(t)}{\partial t},\quad t\in[0,1],\quad \zeta\in[0,\pi]^2,\\ 
			{\mathcal X}(t,\zeta)&=0,\quad t\in[0,1],\quad \zeta\in\partial[0,\pi]^2,\quad\text{and}\quad \mathcal X(0,\zeta) \equiv 0,
		\end{align*}
	where $\alpha, \beta>0$, $\gamma\in\mathbb R$ and $w$ is an one-dimensional Wiener process. ${\mathcal X}(t,\cdot)$, $t\in [0, T]$, is interpreted as a process taking values in $H=L^2([0,\pi]^2)$. The input operator $\mathcal B$ in \eqref{heateq} is characterized by $1_{[\frac{\pi}{4},\frac{3\pi}{4}]^2}(\cdot)$ and the noise operator $\mathcal N_1= \mathcal N$ is defined trough $ \mathcal N\mathcal X =\expn^{-|\cdot-\frac{\pi}{2}|-\cdot}\mathcal X$ for $ \mathcal X\in L^2([0,\pi]^2) $. Since the Dirichlet Laplacian generates a $C_0$-semigroup and its eigenfunctions $(h_k)_{k\in\mathbb N}$ represent a basis of $H$, the same is true for  $\alpha \Delta+\beta I$. Therefore, we interpret the solution of the above SPDE in the mild sense. For more information to SPDEs and the mild solution concept, we refer to \cite{dapratozab}. The quantity of interest is the average temperature on the non controlled area, i.e., 
	\[\mathcal Y(t)= \mathcal C \mathcal X(t, \cdot):=\frac{4}{3\pi^2}\int_{[0,\pi]^2\setminus[\frac{\pi}{4},\frac{3\pi}{4}]^2}\mathcal X(t,\zeta)d\zeta.\]
In order to solve this SPDE numerically, a spatial discretization can be considered as a first step. Here, we choose a spectral Galerkin method relying on the global basis of eigenfunctions $(h_k)_{k\in\mathbb N}$. The idea is to construct an approximation $\mathcal X_n$ to $\mathcal X$ taking values in the subspace $H_n =\spaned \{h_1,\cdots , h_n\}$ and which converges to the SPDE solution with $n\rightarrow \infty$. For more detailed information on this discretization scheme, we refer to \cite{galerkinhaus}. The vector of Fourier coefficients $x(t)= \left( \langle \mathcal X_n(t),h_1\rangle_H, \cdots, \langle \mathcal X_n(t),h_n\rangle_H\right)^\top$ is a solution of a system like \eqref{original_system} with $q=1$ and discretized operators
\begin{itemize}
	\item $ {A}=\alpha\diag(-\lambda_1,\cdots,-\lambda_n) +\beta I$,\quad $ {B}=\left( \langle \mathcal B,h_k\rangle_H\right)_{k=1 \cdots n}$, \quad ${C}=\left(\mathcal C h_k\right)_{
	k=1\cdots n}$, 
	\item $ {N}_1=\left( \langle \mathcal N h_i ,h_k\rangle_H\right)_{k,i=1 \cdots n} $\quad and\quad $ x_0=0$,
\end{itemize}		
where $(-\lambda_k)_{k\in\mathbb N}$ are the ordered eigenvalues of $\Delta$. We refer to \cite{redmannbenner}, where a similar example was studied. There, more details are provided on how this system with its matrices is derived. Now, a small $\alpha$ and a larger $\beta$ yield an unstable $A$, i.e., $\sigma(A)\not\subset \mathbb C$ which already violates asymptotic mean square stability of \eqref{original_system}, i.e., $\mathbb E\left\|x(t; x_0, 0)\right\|_2^2\nrightarrow 0$ as $t\rightarrow \infty$ . Moreover, a larger $\gamma$ (larger noise) causes further instabilities. For that reason, we pick $\alpha= 0.4$, $\beta=3$ and $\gamma = 2$ in order to demonstrate the MOR procedure for a relatively unstable system. Notice that enlarging $\beta$ or $\gamma$ (or making $\alpha$ smaller) leads to a higher degree of instability. This affects the approximation quality in the reduced system given $T$ is fixed. The intuition is that the less stable a system is the stronger the dominant subspaces are expanding in time. This is because some variables in unstable systems are strongly growing such that initially redundant directions become more relevant from a certain point of time. This can also be observed in numerical experiments. \smallskip

Below, we fix a normalized control $u(t)= c_u \expn^{-0.1 t}$, $t\in [0, T]$, (the constant $c_u$ ensures $\left\|u\right\|_{L^2_T}=1$) and apply the MOR method to the spatially discretized SPDE that is based on the balancing transformation $S=S_2$ described in Section \ref{sec_diagonaling}. In Section \ref{exp_n100}, we compare the approximation quality of the ROMs using either the exact Gramian or inexact Gramians introduced in Section \ref{sec_compGram}. Subsequently, Section \ref{exp_n1000} shows the reduced model accuracy in higher state space dimension, where solely inexact Gramians are available. We conclude the numerical experiments by discussing the impact of the terminal time $T$ and the covariance matrix $K$ in Section \ref{different_parameter}.

\subsection{Simulations for $n=100$ and $T=1$}\label{exp_n100}
We compare the associated ROM \eqref{reduced_system} with the original system in dimension $n=100$ first since this choice allows to determine $F(T), G(T)$ and hence the Gramians $P_T, Q_T$ exactly according to Section \ref{exact_Gram}. As a consequence, we can compare the MOR scheme involving the exact Gramians with the same type of scheme relying on the approximated Gramians that are computed exploiting the approaches in Sections \ref{sampled_Gram} and \ref{approxFG_Gram}. In particular, we first approximate $F(T)$ and $G(T)$ based on a Monte-Carlo simulation using $10$ realizations of the estimators \eqref{simulated_est_F} and \eqref{simulated_est_G}, respectively. The functions $\mathcal F$ and $\mathcal G$ entering these estimators are chosen as in \eqref{estimationF} with $c=0$. We refer to the resulting matrices as the Section \ref{sampled_Gram} Gramians. At this point, we want to emphasize that these sampling based Gramians do not necessarily have to be accurate approximations of the exact Gramians in a component-wise sense. It is more important that the dominant subspaces of the system (eigenspaces of the Gramians) are captured in the approximation. Notice that the dominant subspace characterization is not improved if the number of samples is enlarged to $1000$. 
Secondly, we determine the approximations $\mathcal P_T$ and $\mathcal{Q}_T$ according to Section \ref{approxFG_Gram} and call them Section \ref{approxFG_Gram} Gramians. The associated constants are chosen to be $c_F=c_G=0$.\smallskip

In Figure \ref{hsvplot}, the HSVs $\sigma_{T, i}$, $i=\{1, \ldots, 50\}$, of system \eqref{original_system} are displayed. By Theorem \ref{thm_error_bound} and the explanations below this theorem, it is known that small truncated $\sigma_{T, i}$ go along with a small reduction error of the MOR scheme. Due to the rapid decay of these values, we can therefore conclude that small error can already be achieved for small reduced dimensions $r$. For instance, we observe that $\sigma_{T, i}<3.5$e$-06$ for $i\geq 8$ indicating a very high accuracy in the  ROM for $r\geq 7$. This is confirmed by the error plot in Figure \ref{errorK0} and the second column of Table \ref{table_computation2}. Moreover, Figure \ref{errorK0} shows the tightness of the error bound in \eqref{error_1} that was specified in Theorem \ref{thm_error_bound}. The bound differs from the exact error only by a factor between $2.5$ and $4.6$ for the reduced dimensions considered in Figure \ref{errorK0} and is hence a good indicator for the expected performance. Notice that the error is only exact up to deviations occurring due to the semi-implicit Euler-Maruyama discretization of \eqref{original_system} and \eqref{reduced_system} as well as the Monte-Carlo approximation of the expected value using $10\,000$ paths. Besides the MOR error based on $P_T$ and $Q_T$, Table \ref{table_computation2} states the errors in case the approximating Gramians of Sections \ref{sampled_Gram} and  \ref{approxFG_Gram} are used. It can be seen that both approximations perform roughly the same and that one looses an order of accuracy compared to the exact Gramian approach. However, one can lower  the  reduction error by an optimization with respect to the constants $c, c_F, c_G$. Moreover, we see that the accuracy is very good for the estimators of the covariances $F(T)$ and $G(T)$ used here.

\begin{figure}[ht]
\begin{minipage}{0.48\linewidth}
 \hspace{-0.5cm}
\includegraphics[width=0.9\textwidth,height=6cm]{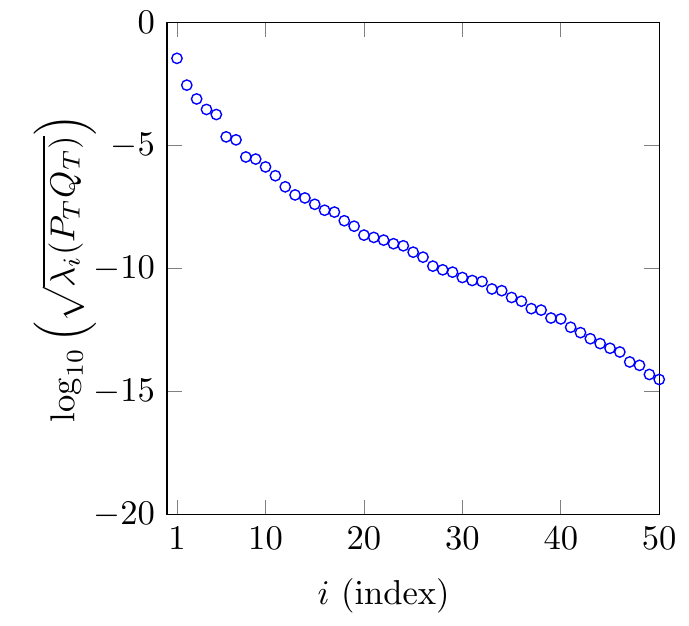}
\caption{Decay of first $50$ logarithmic HSVs of system \eqref{original_system} based on time-limited Gramians $P_T$ and $Q_T$.}\label{hsvplot}
\end{minipage}\hspace{0.5cm}
\begin{minipage}{0.48\linewidth}
 \hspace{-0.5cm}
\includegraphics[width=0.9\textwidth,height=6cm]{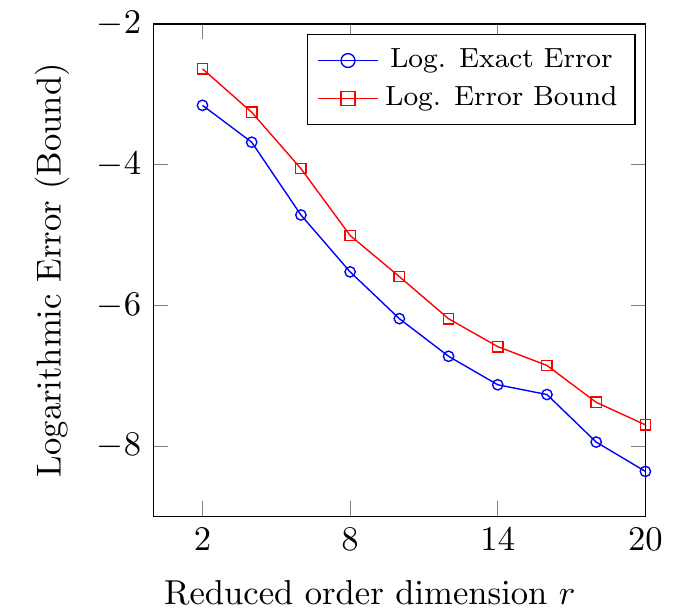}
\caption{$\log_{10}\big(\sup_{t\in[0,1]}\mathbb{E}\|y(t)-\bar{y}(t)\|_2\big)$ and logarithmic bound in \eqref{error_1} for $r\in\{2, 4, 6, 8, 10, 12, 14, 16, 18, 20\}$.}\label{errorK0}
\end{minipage}
\end{figure}

\begin{table}[th]
\centering
\begin{tabular}{|c||c|c|c|}\hline
& \multicolumn{3}{c}{Error $\sup_{t\in[0,1]}\mathbb{E}\|y(t)-\bar{y}(t)\|_2$ of MOR using } \vline \\
\hline 
Reduced dimension $r$ & exact Gramians $P_T, Q_T$ &  Section \ref{sampled_Gram} Gramians &  Section \ref{approxFG_Gram} Gramians\\
\hline
\hline
$2$ & $7.00$e$-04$& $2.61$e$-03$ & $1.75$e$-03$\\
$4$ &  $2.09$e$-04$ & $1.82$e$-03$ & $8.61$e$-04$\\
$8$ & $2.99$e$-06$ & $2.63$e$-05$ & $4.51$e$-05$ \\
$16$ & $5.38$e$-08$ & $1.31$e$-06$ &  $1.55$e$-06$ \\
\hline
\end{tabular}\caption{Error between the output $y$ of \eqref{original_system} with $n=100$ and the reduced output $\bar y$ of \eqref{reduced_system} using different Gramians to compute the balancing transformation $S=S_2$.}
\label{table_computation2}
\end{table}

\subsection{Simulations for $n=1000$ and $T=1$}\label{exp_n1000}
We repeat the simulations of Subsection \ref{exp_n100} for $n=1000$. This is a scenario, where the exact Gramians are not available anymore. Therefore, we conduct the balancing MOR scheme using the Sections \ref{sampled_Gram} and \ref{approxFG_Gram} Gramians only. In the context of the Section \ref{sampled_Gram} Gramians, it is important to mention that in higher dimensions it is required to use very efficient discretizations of the Ito integrals in \eqref{simulated_est_F} and \eqref{simulated_est_G}. Otherwise, a very small step size is needed such that from the computational point of view it is better to omit these Ito integrals within the estimators, i.e., just $x_B$ and $x_C$ are supposed to be sampled to approximate $F(T)$ and $G(T)$. Table \ref{table2} shows that the balancing related MOR technique based on the approximated Gramians of Sections \ref{sampled_Gram} and \ref{approxFG_Gram} is beneficial in high dimensions. A very small reduction error can be observed and in the majority of the cases the sampling based approach seems slightly more accurate than the approach of Section \ref{approxFG_Gram} given the same type of approximations for $F(T)$ and $G(T)$ for each ansatz.
\begin{table}[th]
\centering
\begin{tabular}{|c||c|c|}\hline
& \multicolumn{2}{c}{Error $\sup_{t\in[0,1]}\mathbb{E}\|y(t)-\bar{y}(t)\|_2$ of MOR using } \vline \\
\hline 
Reduced dimension $r$ &  Section \ref{sampled_Gram} Gramians &  Section \ref{approxFG_Gram} Gramians\\
\hline
\hline
$2$ & $1.43$e$-03$ & $1.72$e$-03$\\
$4$ &   $2.07$e$-03$ & $8.57$e$-04$\\
$8$ & $5.18$e$-05$ & $9.26$e$-05$ \\
$16$ & $2.13$e$-06$ &  $4.88$e$-06$ \\
\hline
\end{tabular}\caption{Error between the output $y$ of \eqref{original_system} with $n=1000$ and the reduced output $\bar y$ of \eqref{reduced_system} using Sections \ref{sampled_Gram} and \ref{approxFG_Gram} Gramians to compute the balancing transformation $S=S_2$.}
\label{table2}
\end{table}

\subsection{Relevance of $T$ and $K$}\label{different_parameter}

As in Section \ref{exp_n100}, let us fix $n=100$ to be able to compute the Gramians exactly. We begin with deriving reduced systems on different intervals $[0, T]$. Secondly, we extend our model to a stochastic differential equation with noise dimension $q=2$ and investigate the effect of different correlations between the two Wiener processes.

\paragraph{Relevance of the terminal time}

Let us study the scenario of Section \ref{exp_n100} with $T= 0.5, 1, 2, 3$ using the exact Gramians to illustrate that dominant subspaces are changing in time. Indeed, we observe in Table \ref{table_computation_difT} that for a fixed reduced dimension $r$ the error gets bigger the larger the interval $[0, T]$ is. This means that with increasing $T$ the reduced dimension has to be enlarged to ensure a certain desired approximation error. This is also intuitive in the sense that it is generally harder to find a good approximation on a larger interval in comparison to a smaller one.
\begin{table}[th]
\centering
\begin{tabular}{|c||c|c|c|c|}\hline
& \multicolumn{4}{c}{Error $\sup_{t\in[0, T]}\mathbb{E}\|y(t)-\bar{y}(t)\|_2$ of MOR for } \vline \\
\hline 
Reduced dimension $r$ & $T=0.5$ &  $T=1$ &  $T=2$ & $T=3$\\
\hline
\hline
$2$ & $3.98$e$-04$& $7.00$e$-04$ & $2.17$e$-02$ &$3.13$e$-02$\\
$4$ & $1.46$e$-05$&  $2.09$e$-04$ &  $2.86$e$-04$ &$6.86$e$-04$\\
$8$ & $2.82$e$-07$&$2.99$e$-06$ &  $7.80$e$-06$& $2.23$e$-05$ \\
$16$ & $5.46$e$-09$&$5.38$e$-08$ &   $1.12$e$-07$& $2.90$e$-07$ \\
\hline
\end{tabular}\caption{Error between the output $y$ of \eqref{original_system} and the reduced output $\bar y$ of \eqref{reduced_system} using the exact Gramians:  $n=100$,  $S=S_2$ and $T=0.5, 1, 2, 3$.}
\label{table_computation_difT}
\end{table}

\paragraph{Relevance the the covariance structure}

Let us extend the SPDE discretization by introducing $N_2:= N_1^{\frac{6}{5}}$ so that we have a system of the form \eqref{original_system} with $q=2$ and standard Wiener processes $w_1$ and $w_2$. The goal is to investigate how the correlation between $w_1$ and $w_2$ influences the MOR error. For that reason, we choose the following three scenarios: $\mathbb E[w_1(t) w_2(t)] = \rho t$ with $\rho = 0, 0.5, 1$. Table \ref{table_computation_difK} states the MOR errors for these correlations. In this example, we can observe that a higher correlation between the processes yields a larger error. A different observation was made in numerical examples studied in \cite{mor_heston}, where systems with high correlations in the noise processes gave a smaller reduction error. However, \cite{mor_heston} studies  different types of stochastic differential equations  in the context of asset price models which do not have control inputs.
\begin{table}[th]
\centering
\begin{tabular}{|c||c|c|c|}\hline
& \multicolumn{3}{c}{Error $\sup_{t\in[0,1]}\mathbb{E}\|y(t)-\bar{y}(t)\|_2$ of MOR for } \vline \\
\hline 
Reduced dimension $r$&\textcolor{white}{fffff}$\rho=0$ \textcolor{white}{fffff}&$\textcolor{white}{fff}\rho=0.5$\textcolor{white}{fff}&$\textcolor{white}{ffff}\rho=1$\textcolor{white}{ffff}\\
\hline
\hline
$2$ & $1.10$e$-03$& $1.43$e$-03$ & $1.79$e$-03$\\
$4$ &  $2.44$e$-04$ & $2.34$e$-04$ & $3.24$e$-04$\\
$8$ & $5.71$e$-06$ & $8.95$e$-06$ & $1.34$e$-05$ \\
$16$ & $1.64$e$-07$ & $2.37$e$-07$ &  $3.36$e$-07$ \\
\hline
\end{tabular}\caption{Error between the output $y$ of \eqref{original_system} and the reduced output $\bar y$ of \eqref{reduced_system} using the exact Gramians:  $n=100$,  $S=S_2$, $T=1$, $q=2$ and different correlations $\rho=0, 0.5, 1$.}
\label{table_computation_difK}
\end{table}

\bibliographystyle{plain}

\end{document}